\documentclass[11pt]{article}
\usepackage{amsmath,amsfonts, amssymb,amsthm,multicol,graphicx,tikz, enumitem,asymptote}
\usepackage{makeidx}

\newcommand{\ds}{\displaystyle}

\newcommand{\ZZ}{\mathbb{Z}}

\newcommand{\RR}{\mathbb{R}}

\newcommand{\cF}{\mathcal{F}}

\newcommand{\cE}{\mathcal{E}}

\newcommand{\EE}{\mathbb{E}}

\newcommand{\CR}[1]{\textcolor{red}{#1}}

\newtheorem{prop}{Proposition}   
\newtheorem{lem}[prop]{Lemma}
\newtheorem{thm}[prop]{Theorem}

\newtheorem{corollary}[prop]{Corollary}

\usetikzlibrary{calc}

\title{Behavior of the Minimum Degree Throughout the $d$-process}
\author{Jakob Hofstad\thanks{This work was supported by the National Science Foundation (NSF) under research grant DMS-2246907. The NSF had no role in the production or publication of this manuscript.}
\\ Carnegie Mellon University
}

\begin{document}

\maketitle

\begin{center}
    {\bf Abstract}
\end{center}

The $d$-process generates a graph at random by starting with an empty graph with $n$ vertices, then adding edges one at a time uniformly at random among all pairs of vertices which have degrees at most $d-1$ and are not mutually joined. We show that, in the evolution of a random graph with $n$ vertices under the $d$-process with $d$ fixed, with high probability, for each $j \in \{0,1,\dots,d-2\}$, the minimum degree jumps from $j$ to $j+1$ when the number of steps left is on the order of $\ln(n)^{d-j-1}$. This answers a question of Ruci\'nski and Wormald. More specifically, we show that, when the last vertex of degree $j$ disappears, the number of steps left divided by $\ln(n)^{d-j-1}$ converges in distribution to the exponential random variable of mean $\frac{j!}{2(d-1)!}$; furthermore, these $d-1$ distributions are independent.

\section{Introduction}

There are numerous models that generate different types of sparse random graphs. Among them is the \textit{$d$-process}, defined in the following way: start with $n$ vertices and 0 edges, and at each time step, choose a pair $\{u,v\}$ uniformly at random over all pairs consisting of vertices which have degree less than $d$ and are not joined to each other by an edge. $d$ could be allowed to change with $n$, but for the rest of this paper $d$ is always a fixed constant (this is also assumed in all relevant citations). Ruci{\'n}ski and Wormald showed that with high probability, abbreviated ``w.h.p." (i.e., with probability converging to 1 as $n \to \infty$) the $d$-process ends with $\lfloor dn/2 \rfloor$ edges \cite{d-regular-likely}. There is still much that is unknown about the $d$-process; for example, it is not known whether the $d$-process is {\it contiguous} with the $d$-uniform random graph model for any $d \geq 2$; i.e., if any event that happens with high probability in one happens with high probability in the other. A recent paper by Molloy, Surya, and Warnke \cite{NotUniform} disproves this relation if there is enough ``non-uniformity" of the vertex degrees (with an appropriate modification of the $d$-process for non-regular graphs); it also contains a good history of the $d$-process. See \cite[Section 9.6]{JSRRandom} for more on contiguity.

A couple of notable results have been given for the case where $d = 2$: the expected numbers of cycles of constant sizes were studied by Ruci{\'n}ski and Wormald in \cite{2-process-short-cycles}, and in \cite{Hamiltonicity}, Telcs, Wormald, and Zhou calculated the probability that the 2-process ends with a Hamiltonian cycle. In these works, the authors establish estimates on certain graph parameters, such as the number of vertices of a certain degree, that hold throughout the process. This is done with the so-called  ``differential equations method" for random graph processes, which uses martingale inequalities to give variable bounds; in \cite{DiffQMethod} Wormald gives a thorough description of this method. 

More recently, Ruci\'nski and Wormald  announced a new analysis of the $d$-process that hinges on a coupling with a balls-in-bins process. This simple argument gives a precise estimate of the probability that the $d$-process ends with  $\lfloor dn/2 \rfloor$ edges (i.e., the probability that the $d$-process reaches saturation). This argument includes estimates for the number of vertices of each degree near the end of the process. This work was presented by Ruci\'nski at the 2023 \textit{Random Structures and Algorithms} conference. Ruci\'nski's presentation included the following problem (which was open at the time): when do we expect the last vertex of degree $j$ (for any $j$ from 0 to $d-2$) to disappear? The question was also stated earlier for $d=2$ and $j=0$ by Ruci{\'n}ski and Wormald \cite[Question 3]{2-process-short-cycles}. In November of 2023, after the first release of the pre-print of this paper, Ruci{\'n}ski and Wormald released a pre-print of their balls-and-bins argument which also included an answer to Ruci{\'n}ski's question \cite{d-process-balls-bins}. Our main result uses the differential equations method (as described in the previous paragraph) and gives a slightly stronger answer:

\begin{thm} \label{Theorem: main}
    Consider the d-process on a vertex set of size n, and for each $\ell \in \{0\} \cup [d-2]$, let the random variable 
    $T_{\ell}$ be the step at which the number of vertices of degree at most $\ell$ becomes 0. Then the sequence (over $n$) of random $d-1$-tuples consisting of the variables $$V_{n}^{(\ell)} = \frac{(d-1)!(dn-2T_{\ell})}{\ell!(\ln(n))^{d-1-\ell}}$$

    converges in distribution to the product of $d-1$ independent exponential random variables of mean 1. 
\end{thm}

In this paper we use the differential equations method with increasingly precise estimates of certain random variables; these estimates are known as {\it self-correcting}. Previous results that use self-correcting estimates include \cite{Hamiltonicity}, \cite{SIR}, \cite{GreedyTriangle}, \cite{RandomTriangle}, \cite{DynamicTriangle}, and \cite{TriangleRamsey}. There have been various approaches to achieving self-correcting estimates; the approach in this paper uses {\it critical intervals}, regions of possible values of a random variable in which we expect subsequent variables to increase/decrease over time. Critical intervals used in this fashion first appeared in a result of Bohman and Picollelli \cite{SIR}. For an introduction to and discussion of the method see Bohman, Frieze, and Lubetzky \cite{GreedyTriangle}.

The proof of Theorem \ref{Theorem: main} is divided into four sections. In Section \ref{Section: Preliminaries} we introduce random variables of the form $S^{(j)}_i$ which count the number of vertices of degree at most $j$ after $i$ steps, define approximating functions $s_{j}(t)$ with the eventual goal of showing that $S^{(j)}_{i} \approx ns_{j}(i/n)$ for most of the process, and derive useful properties of these functions. One such property is that, when there are at most $n^{c}$ steps left for some constant $c < 1$, $$\frac{s_{j}(i/n)}{s_{j-1}(i/n)} = \Theta(\ln(n))$$ for each $j$; this hierarchy of functions helps us to focus on each variable $S^{(j)}_i$ independently of the others when it is near 0, which motivates the form of the limiting exponential random variables in Theorem \ref{Theorem: main}. At the end of Section \ref{Section: Preliminaries} we introduce two martingale inequalities used by Bohman \cite{Tri-free} and make a slight modification to them to use later in the paper. In Section \ref{Section: First} we work with a more `standard martingale method' (without the use of critical intervals) to show that $S^{(j)}_{i} \approx ns_{j}(i/n)$ until there are $n^{1 - 1/(100d)}$ steps left. Here we allow the error bounds to increase over time. In Section \ref{Section: Second} we use a more refined martingale method (including the use of critical intervals) to show that $S^{(j)}_{i} \approx ns_{j}(i/n)$ continues to hold until there are $\ln (n) ^{d-0.8-j}$ steps left; here the error bounds {\it decrease} over time, and so are self-correcting. In Section 5 we complete the proof of Theorem \ref{Theorem: main} by using a pairing argument to show that, in the last steps of the $d$-process, the behavior of the random variable in question can be well-approximated by a certain uniform distribution of time steps. Sections \ref{Section: Second} and \ref{Section: Final} are both parts of a proof by induction over a series of intervals of time steps, though we give each part its own section as the methods used in each are very different.

\section{Preliminaries} \label{Section: Preliminaries}

First, two technical notes: we use the standard notation of symbols $o, O, \Theta, \omega, \Omega, \ll, \gg$, and $\sim$ to compare functions asymptotically (e.g., see pages 9-10 of \cite{JSRRandom}). We also note that, throughout the paper, we assume $n$ to be arbitrarily large.

In this section we set up sequences of random variables, describe how the evolution of the $d$-process depends on these, and deduce properties of certain \textit{approximating functions}; such functions are used to estimate the number of vertices of given degrees throughout the process (much of this is also described in \cite{Hamiltonicity} with similar notation; the one major difference is that we use $i$ instead of $t$ for the number of time steps, and $t$ instead of $x$ for the corresponding time variable). Consider a sequence of graphs $G_0,G_1,\dots,G_{\lfloor dn/2 \rfloor}$, where $G_0$ is the empty graph of $n$ vertices, and for $i \in [n]$, let $G_i$ be formed by adding an edge uniformly at random to $G_{i-1}$ so that the maximum degree of $G_i$ is at most $d$ (in the unlikely event that there are no valid edges to add after $s$ steps for some $s < \lfloor dn/2 \rfloor$, let $G_i = G_{s}$ for all $i > s$).  Next, we define several sequences of random variables: For all $i, j, j_1, j_2$ such that $0 \leq i \leq \lfloor \frac{dn}{2} \rfloor$, $0 \leq j \leq d$, and $0 \leq j_1 \leq j_2 \leq d-1$, define: \\ \newline
\phantom{.}\qquad $Y^{(j)}_i := $ the number of vertices in $G_i$ with degree $j$ \newline
\phantom{.}\qquad $S^{(j)}_i := $ the number of vertices in $G_i$ with degree {\it at most} $j$   \newline
\phantom{.}\qquad $Z^{(j_1,j_2)}_i := $ the number of edges $\{v_1,v_2\}$ in $G_i$ for which $$\min \{\deg(v_1), \deg(v_2) \} = j_1 \ \text{ and } \ \max \{\deg(v_1), \deg(v_2)\} = j_2$$
\phantom{.}\qquad $Z_i := \ds\sum_{0 \leq j_1 \leq j_2 \leq d-1} Z^{(j_1,j_2)}_i$.\\

By definition and by edge-counting we have 
    $$\sum_{j=0}^{d} Y_i^{(j)} = n \qquad \text{and} \qquad \sum_{j=1}^{d} j Y_i^{(j)} = 2i.$$
Combining the equations above gives us

\begin{equation}
    \sum_{j=0}^{d-1} S_i^{(j)} = \sum_{j=0}^{d-1} (d-j) Y_i^{(j)} = dn-2i. \label{equ: Y sum}
\end{equation}

One can also quickly verify from (\ref{equ: Y sum}) and by definition of $S^{(d-1)}_i$, $Z^{(j_1,j_2)}_{i}$, and $Z_i$, that

\begin{equation}
    d S^{(d-1)}_i \geq \max \{2 Z_i, dn - 2i\} \qquad \text{and} \qquad j Y^{(j)}_i \geq \sum_{k \leq j} Z^{(k,j)}_i + \sum_{k \geq j} Z^{(j,k)}_i.  \label{equ: U, Z bounds}
\end{equation}

Throughout the process we will keep track of the variables $S^{(j)}$ using martingale arguments. This is sufficient for us, as the $Y$ variables can be derived from the $S$ variables, and because none of the $Z$ variables will have any significant effect in any of our calculations, as we will see later. \\

Our next step is to estimate the expected on-step change of $S^{(j)}_i$, known as the ``trend hypothesis" in \cite{DiffQMethod}. Note that $S^{(j)}_{i} - S^{(j)}_{i+1}$ equals the number of vertices of degree $j$ that are picked at the $i+1$ time step; hence, for all $j \in \{0\} \cup [d-1]$:

\vspace{-4mm}

\begin{align}
    \EE[S^{(j)}_{i+1} - S^{(j)}_{i} \mid G_i] &= \frac{-Y^{(j)}_i (S^{(d-1)}_i - 1) + \ds\sum_{k \leq j} Z^{(k,j)}_i + \sum_{k \geq j} Z^{(j,k)}_i }{\binom{S^{(d-1)}_i}{2} - Z_i} \nonumber \\&=
    \frac{-2Y^{(j)}_i}{S^{(d-1)}_i}\left(1 + O\left(\frac{1}{dn-2i}\right)\right) \label{equ: mult error} \qquad \text{by (\ref{equ: U, Z bounds})}
    \\&=
    \frac{-2 Y^{(j)}_i}{S^{(d-1)}_i}+ O\left(\frac{1}{dn-2i}\right). \label{equ: trend hypothesis}
\end{align}

For $j \in \{0\} \cup [d-1]$ we define approximating functions $y_j:[0,d/2) \to \RR$ and $s_j:[0,d/2) \to \RR$: let $y_j(t)$, $s_j(t)$ be functions such that $s_j = \sum_{k=0}^{j} y_k$, $y_0(0) = 1$ and $y_k(0) = 0$ for all $k \in [d-1]$ (equivalent to $s_j(0) = 1$ for all $j$), and (assuming the ``dummy functions" $y_{-1}(t) = s_{-1}(t) = 0$):

\begin{equation}
     \frac{ds_j}{dt} = \frac{-2y_{j}}{s_{d-1}} = \frac{2(s_{j-1} - s_j)}{s_{d-1}} \ \qquad \ \frac{dy_j}{dt} = \frac{2(y_{j-1} - y_j)}{s_{d-1}}.  \label{equ: y deriv}
\end{equation}

By (\ref{equ: y deriv}) and the chain rule, for all $j \in [d-1]$: $$\frac{d y_j}{d y_0} - \frac{y_j}{y_0} = -\frac{y_{j-1}}{y_0}.$$
Since the above equation is first-order linear, we have, for some constant $C_j$:
    $$y_j = y_0\left(-\int \frac{y_{j-1}}{y_0^2} d y_0 + C_j\right).$$
Using the above recursively with initial conditions, we have, for all $j \in [d-1]$:

\begin{equation}
    y_j = \frac{y_0 (-\ln(y_0))^j}{j!}. \label{equ: y relation}
\end{equation}

To solve for an explicit formula relating $y_0$ and $t$, note that, by (\ref{equ: y deriv}): $$\frac{d}{dt} \left( \sum_{j=0}^{d-1} (d-j) y_j \right)= \frac{-2 \ds\sum_{j=0}^{d-1}y_j}{s_{d-1}} = -2,$$
hence, using initial conditions (note the resemblance to (\ref{equ: Y sum})):

\begin{equation}
    \sum_{j=0}^{d-1} s_j = \sum_{j=0}^{d-1} (d-j) y_j = d - 2t. \label{equ: y explicit}
\end{equation}

Equations (\ref{equ: y relation}) and (\ref{equ: y explicit}) together give us a complete description of the functions $y_j$ and $s_j$. We will now prove some useful properties of these functions. To start, we can combine (\ref{equ: y relation}) and (\ref{equ: y explicit}) to get

\begin{equation}
    \sum_{j=0}^{d-1} \frac{y_0 (-\ln(y_0))^j (d-j)}{j!}  = d - 2t. \label{equ: y_0 explicit}
\end{equation}

Note that, by continuity of $y_0$ and by (\ref{equ: y_0 explicit}), $y_0 > 0$ over its domain. Next, by summing up (\ref{equ: y relation}) over $j \in \{0\} \cup [d-1]$ ((\ref{equ: y relation}) holds for $j=0$ also) one can see that $s_{d-1}$ is positive if $y_0 \leq 1$. This, combined with $\frac{d y_0}{dt} = \frac{-2y_0}{s_{d-1}}$, tells us that $y_0$ is decreasing and $s_{d-1}$ is positive. It follows from $y_0 \in (0,1]$ and (\ref{equ: y relation}) that each $y_j$ is positive for $t \not= 0$. In turn, this implies that $0 \leq y_j \leq s_j \leq s_{d-1}$ for each $j$. From this it follows that $ds_{d-1}$ is at least the left expression of (\ref{equ: y explicit}), so $s_{d-1} \geq 1 - \frac{2t}{d}$. We make a special note of the last couple of properties mentioned:

\begin{equation}
    0 \leq y_j \leq s_j \leq s_{d-1} \text{ for all $j$ \qquad and} \qquad s_{d-1}(i/n) \geq 1 - \frac{2i}{dn}. \label{equ: s_{d-1} bound}
\end{equation}

Next, we want to understand the behavior of each function when $t$ is close to $\frac{d}{2}$, as this is the most critical point of the process. Consider (\ref{equ: y_0 explicit}) again. As $t \to \frac{d}{2}, y_0 \to 0$, so $\frac{y_0(- \ln(y_0))^{d-1}}{(d-1)!}$ will be the most dominant term on the left; hence, $$
    t \to \frac{d}{2} \ \Longrightarrow \ y_0 \sim \frac{(d-1)! (d-2t)}{(-\ln(d-2t))^{d-1}}.$$ This, combined with (\ref{equ: y relation}) gives us, for all $j \in \{0 \} \cup [d-1]$:

\begin{equation}
    t \to \frac{d}{2} \ \Longrightarrow \ y_j(t) \sim s_j(t) \sim \frac{(d-1)!(d-2t)}{j!(-\ln(d-2t))^{d-1-j}} \label{equ: y approx initial}.
\end{equation}

For large enough $t$ (and hence, for a large enough step $i$), we can approximate the above expression:

\begin{equation}
    i \geq \frac{dn}{2} - n^{1 - 1/(100d)} \ \Longrightarrow \ ny_j(i/n) \sim ns_j(i/n)  = \Theta\left(\ln(n)^{-d+1+j}\left(
    \frac{dn}{2}-i\right) \right)\label{equ: y approx}.
\end{equation}

One can now see that, near the end of the process, $s_j/s_{j-1} = \Theta(\ln(n))$, as mentioned in the introduction.

Finally, we introduce two martingale inequalities from a result of Bohman \cite{Tri-free} which will be used in Section \ref{Section: Second} in a slightly modified form. The original inequalities are as follows:

\begin{lem}[Lemma 6 from \cite{Tri-free}] \label{Lemma: upper bound}
Suppose $a, \eta,$ and $N$ are positive, $\eta \leq N/2$, and $a < \eta m$. If $0 = A_0, A_1,\dots, A_m$ is a submartingale such that $-\eta \leq A_{i+1} - A_{i} \leq N$ for all $i$, then $$
    \mathbb{P}[A_m \leq -a] \leq e^{-\frac{a^2}{3\eta N m}}.
$$

\end{lem}

\begin{lem}[Lemma 7 from \cite{Tri-free}] \label{Lemma: lower bound}
Suppose $a, \eta,$ and $N$ are positive, $\eta \leq N/10$, and $a < \eta m$. If  $0 = A_0, A_1,\dots, A_m$ is a supermartingale such that $-\eta \leq A_{i+1} - A_{i} \leq N$ for all $i$, then $$
    \mathbb{P}[A_m \geq a] \leq e^{-\frac{a^2}{3\eta N m}}.
$$

\end{lem}
We present the following modification, which removes the requirement $a < \eta m$ and modifies one of the inequalities slightly:

\begin{corollary} \label{Cor: upper bound mod}
Suppose $a, \eta,$ and $N$ are positive, and $\eta \leq N/2$. If $0 = A_0, A_1,\dots, A_m$ is a submartingale such that $-\eta \leq A_{i+1} - A_{i} \leq N$ for all $i$, then $$
    \mathbb{P}[A_m \leq -a] \leq e^{-\frac{a^2}{3\eta N m}}.$$

\end{corollary}

\begin{corollary} \label{Cor: lower bound mod}
Suppose $a, \eta,$ and $N$ are positive, and $\eta \leq N/10$. If  $0 = A_0, A_1,\dots, A_m$ is a supermartingale such that $-\eta \leq A_{i+1} - A_{i} \leq N$ for all $i$, then $$
    \mathbb{P}[A_m \geq a] \leq e^{-\frac{a^2}{3\eta N m}} + e^{-\frac{a}{6N}}.$$

\end{corollary}
Corollary \ref{Cor: upper bound mod} is nearly immediate from Lemma \ref{Lemma: upper bound}: first, one can extend the result to include $a = \eta m$ by using left-continuity (with respect to $a$) of both sides of the inequality; we hence assume $a > \eta m$. Since $A_{i+1} - A_i \geq -\eta > -a/m$, then $A_m = A_m - A_0 > -a$. We now derive Corollary \ref{Cor: lower bound mod} from Lemma \ref{Lemma: lower bound}: assume $a \geq \eta m$, and let $m' \in \ZZ^+$ such that $a < \eta m' \leq 2a$. Extend the martingale by adding variables $A_{m+1},\dots,A_{m'}$ which are all equal to $A_m$. Apply Lemma \ref{Lemma: lower bound} with $m$ replaced with $m'$, and use $\eta m' \leq 2a$ to get
    $$\mathbb{P}[A_{m} \geq a] = \mathbb{P}[A_{m'} \geq a] \leq e^{-\frac{a}{6N}}.$$
Combining the case $a < \eta m$ from Lemma \ref{Lemma: lower bound} and the case $a \geq \eta m$ above gives Corollary \ref{Cor: lower bound mod}.

\section{First phase} \label{Section: First}

Let $i_{trans} = \lfloor \frac{dn}{2} - n^{1 - 1/(100d)} \rfloor$. The objective of this section is to prove the following Theorem:

\begin{thm} \label{Theorem: 1st phase bounds}

Define $E_{first}(i) := n^{0.6} \left(\frac{dn}{dn-2i}\right)^{4d}$. With high probability, for all $i \leq i_{trans}$ and all $j \in \{ 0 \} \cup [d-1]$:

\begin{equation}
    \left|S^{(j)}_{i} - n s_j\left(\frac{i}{n}\right)\right| \leq E_{first}(i). \label{equ: 1st Phase Theorem bound}
\end{equation}
\end{thm}

Now we define two new random variables for each $j$:

\vspace{-5mm}

\begin{align*}
    S^{(j)+}_{i} &:= S^{(j)}_i - n s_j(i/n) - E_{first}(i)\\
    S^{(j)-}_{i} &:= S^{(j)}_i - n s_j(i/n) + E_{first}(i).
\end{align*}

Next, we introduce a stopping time $T$, defined as the first step $i \leq i_{trans}$ for which (\ref{equ: 1st Phase Theorem bound}) is {\it not} satisfied for some $j$; if (\ref{equ: 1st Phase Theorem bound}) always holds, then let $T = \infty$. Although this stopping time is not necessarily needed to prove Theorem \ref{Theorem: 1st phase bounds}, it does make some calculations easier, and moreover, a similar stopping time {\it will} be necessary for the following section; hence, this serves as a good warm-up. Let variable name $W$ be introduced to equip this stopping time to variable $S$, i.e.:

\vspace{-5mm}

\begin{align*}
    &W^{(j)+}_{i} := \begin{cases}
        &S^{(j)+}_{i}, i < T \\
        &S^{(j)+}_{T}, i \geq T
    \end{cases}
    &W^{(j)-}_i := \begin{cases}
        &S^{(j)-}_i, i < T \\
        &S^{(j)-}_T, i \geq T.
    \end{cases}
\end{align*}

\vspace{-4mm}

Note that $W^{(j)+}_{i}$ corresponds to the upper  boundary and $W^{(j)-}_i$ to the lower one in the sense that crossing the corresponding boundary will make the corresponding variable change signs; furthermore, the inequality of Theorem \ref{Theorem: 1st phase bounds} holds if and only if $W^{(j)+}_{i_{trans}} \leq 0$ and  $W^{(j)-}_{i_{trans}} \geq 0$ for each $j$. We now state our martingale Lemma:\\

\begin{lem} \label{Lemma: first phase sub-sup}
    Restricted to $i \leq i_{trans}$, for all $j, \  (W^{(j)-}_{i})_i$ is a submartingale and $(W^{(j)+}_{i})_i$ is a supermartingale.
\end{lem}

\begin{proof}
Here we just prove the first part of the Lemma; the second part follows from nearly identical calculations. Fix some arbitrary $i \leq i_{trans}$; we need to show that  $$\EE[W^{(j)-}_{i+1} - W^{(j)-}_{i} \mid G_i] \geq 0.$$ Also assume that $T \geq i+1$, else $W^{(j)-}_{i+1} - W^{(j)-}_i = 0$ and we are done; it follows that $W^{(j)-}_i = S^{(j)-}_i, W^{(j)-}_{i+1} = S^{(j)-}_{i+1}$, and (\ref{equ: 1st Phase Theorem bound}) holds for the fixed $i$. By (\ref{equ: trend hypothesis}) and (\ref{equ: y deriv}) and using Taylor's Theorem, we have, for some $\psi \in [i,i+1]$:

\begin{align*}
    \EE[S^{(j)-}_{i+1} - S^{(j)-}_i \mid G_i] &= \frac{-2 Y^{(j)}_i}{S^{(d-1)}_i} + O\left(\frac{1}{dn-2i}\right) + \frac{2y_j(i/n)}{s_{d-1}(i/n)} -\frac{d^2}{d\mu^2}\left(\frac{ns_j(\mu/n)}{2}\right)\Big|_{\mu = \psi} \\&+
    (E_{first}(i+1) - E_{first}(i)).
\end{align*}

We split the above expression (excluding $O\left(\frac{1}{dn-2i}\right)$) into three summands.

\begin{enumerate}
    \item Here we make use of the fact that $Y^{(j)}_i = S^{(j)}_i - S^{(j-1)}_{i}$ and $y_j(t) = s_j(t) - s_{j-1}(t)$. We have (putting $s_{j-1}$ and $S_i^{(-1)} = 0$):

\vspace{-5mm}
    
    \begin{align*}
        \frac{-2 Y^{(j)}_i}{S^{(d-1)}_i} + \frac{2y_j(i/n)}{s_{d-1}(i/n)} &= 
        \frac{-2 S^{(j)}_i + 2ns_j(i/n) + 2 S^{(j-1)}_i - 2ns_{j-1}(i/n)}{n s_{d-1}(i/n)} \\ & \ \ \ + 2Y^{(j)}_i \left(\frac{1}{n s_{d-1}(i/n)} - \frac{1}{S^{(d-1)}_i}\right) \\&\geq
        \frac{-4 E_{first}(i)}{ns_{d-1}(i/n)} + 2Y^{(j)}_i \left(\frac{1}{n s_{d-1}(i/n)} - \frac{1}{S^{(d-1)}_i}\right) \qquad \text{by (\ref{equ: 1st Phase Theorem bound}) and $i < T$} \\ & \geq 
        \frac{-4 E_{first}(i)}{ns_{d-1}(i/n)} - \frac{2 Y^{(j)}_{i} E_{first}(i)}{S^{(d-1)}_i (n s_{d-1}(i/n))} \qquad \text{by (\ref{equ: 1st Phase Theorem bound}) and $i < T$}\\ & \geq
        \frac{-6 E_{first}(i)}{n s_{d-1}(i/n)}.
    \end{align*}

    \item 
    \begin{align*}
        -\frac{d^2}{d\mu^2}\left(\frac{ns_j(\mu/n)}{2}\right)\Big|_{\mu = \psi} &= \frac{2}{n}\left(\frac{s_{d-1}(\psi/n)(y_{j-1}(\psi/n) - y_j(\psi/n))+y_j(\psi/n)y_{d-1}(\psi/n)}{(s_{d-1}(\psi/n))^3} \right) \\&=
    O\left(\frac{1}{dn-2i}\right) \qquad \text{by (\ref{equ: s_{d-1} bound})}.
    \end{align*}

    \item For some $\phi \in [i,i+1]$:

\vspace{-5mm}

    \begin{align}
        E_{first}(i+1) - E_{first}(i) &= \frac{dE_{first}(\mu)}{d\mu} \Big|_{\mu = \phi}  \nonumber \\&= 
        8d^{4d+1} n^{4d+0.6} \left(dn-2\phi\right)^{-4d-1} \nonumber \\&=
        (1 + o(1))\frac{8dE_{first}(i)}{dn - 2i}.
        \label{equ: E dif}
    \end{align}

\end{enumerate}

Now we put the three bounds together:

\begin{align*}
    \EE[Y_{i+1}^- - Y_i^- \mid G_i] &\geq \frac{7dE_{first}(i)}{dn - 2i} - \frac{6 E_{first}(i)}{n s_{d-1}(i/n)} + O\left(\frac{1}{dn-2i}\right)
     \\&\geq
    \frac{dE_{first}(i) + O(1)}{dn - 2i} \qquad \text{by (\ref{equ: s_{d-1} bound})} \\&\geq 0.
\end{align*}

\vspace{-10mm}

\end{proof}

Next, we need a Lipschitz condition on each of our variables. Note that $S^{(j)}_{i+1} - S^{(j)}_{i}$ is either $-2, -1$, or 0; also, one can quickly verify that $|s_j((i+1)/n) - s_j(i/n)| \leq \frac{2}{n}$ by (\ref{equ: y deriv}) and (\ref{equ: s_{d-1} bound}), and $|E_{first}(i+1) - E_{first}(i)| = o(1)$ by (\ref{equ: E dif}). Hence, we have, for all $i \leq i_{trans}$ and all $j$:

\begin{equation}
    \max\{|W^{(j)+}_{i+1} - W^{(j)+}_{i}|, |W^{(j)-}_{i+1} - W^{(j)-}_{i}| \} \leq 5.\label{equ: 1st phase Lipschitz}
\end{equation}

We conclude the proof of Theorem \ref{Theorem: 1st phase bounds} by noting that, by Lemma \ref{Lemma: first phase sub-sup} and (\ref{equ: 1st phase Lipschitz}), we can use the standard Hoeffding-Azuma inequality for martingales (e.g. Theorem 7.2.1 in \cite{probmethod}) to show that $\mathbb{P}[W^{(j)+}_{i_{trans}} > 0]$ and $\mathbb{P}[W^{(j)-}_{i_{trans}} < 0]$ are both $o(1)$. For example, for the variable $W^{(j)+}_{i}$ one would get $$
    \mathbb{P}[W^{(j)+}_{i_{trans}} > 0] \leq \exp \left\{ - \frac{n^{1.2}}{50 i_{trans}} \right\} = o(1).
$$

\section{Second Phase} \label{Section: Second}

The second phase is where the more sophisticated tools will be used, including the use of critical intervals, self-correcting estimates, and a more general martingale inequality. Furthermore, this phase is broken up into $d-1$ sub-phases, in relation to when each of the $d-1$ sequences $S^{(j)}$ (for $j \leq d-2$) terminate at 0. First, a few definitions: for all $k  \in \{0\} \cup [d-2]$, define $$
    i_{after}(k) := \left\lfloor \frac{dn}{2} - \ln(n)^{d-1.01-k} \right\rfloor.$$
These step values will govern the endpoints of the sub-phases: define for all $k \in \{0\} \cup [d-2]$: $$
    I_k := \begin{cases}
        [i_{trans} + 1, i_{after}(0)], & k = 0 \\
        [i_{after}(k-1) + 1, i_{after}(k)], & k > 0.
    \end{cases}$$

Next, for all $i,j,k$ such that $0 \leq j < d$, $0 \leq k < d-1$, and $i \in I_k$, define error functions $$
    E_{j,k}(i) = E_{j}(i) := 2^k \ln (n)^{0.05} (n s_{j}(i/n))^{0.7}.$$
Note that, by (\ref{equ: y approx}), we have

\begin{equation}
    E_j(i) = \Theta \left( \ln(n)^{-0.7d + 0.75 + 0.7j} \left(\frac{dn}{2} - i\right)^{0.7}\right).\label{equ: E approx}
\end{equation}

Finally, for any $r \in \RR_+$ and $\ell \in [d-2]$, define $$
    i(r,\ell) =  \frac{dn}{2} - \left(\frac{\ell! }{2(d-1)!}\right) r(\ln n)^{d-1-\ell}.$$

The following Theorem will be proved by induction over the $d-1$ sub-phases governed by the index $k$:

\begin{thm} \label{Theorem: 2nd phase bounds}

For each $k \in \{0\} \cup [d-2]$:

\begin{enumerate}
    
    \item With high probability, for all integers $j \in [0, d-1]$ and $i \in I_k$:

    \begin{equation}
        \left|S^{(j)}_{i} - n s_j\left(\frac{i}{n}\right)\right| \leq 4 E_j(i). \label{equ: 2nd phase Theorem bound}
    \end{equation}

    \item $S^{(k)}_{i_{after}(k)} = 0$ with high probability. Furthermore, for any $k+1$-tuple $\{r_0,r_1,\dots,r_k\} \in (\RR_+ \cup \{0\})^{k+1}$:$$
        \Pr \left( \bigcap_{\ell = 0}^{k} \left(S^{(\ell)}_{\lfloor i(r_{\ell},\ell) \rfloor} = 0 \right)\right) \to \exp \left\{-\sum_{\ell=0}^{k} r_{\ell} \right\}.$$

\end{enumerate}

\end{thm}

In the end, it is only the second statement with $k = d-2$ that matters for proving Theorem \ref{Theorem: main}. We make the connection here:

\begin{proof}[Proof of Theorem \ref{Theorem: main} from Theorem \ref{Theorem: 2nd phase bounds}] 

First, note that $S^{(\ell)}_{\lfloor i(r_{\ell},\ell) \rfloor} = 0$ is the same as $T_{\ell} \leq i(r_{\ell},\ell)$, hence by Theorem \ref{Theorem: 2nd phase bounds}: $$
    \mathbb{P} \left( \bigcap_{\ell = 0}^{d-2} \left(T_{\ell} \leq i(r_{\ell},\ell)\right)\right) \to \exp \left\{-\sum_{\ell=0}^{d-2} r_{\ell} \right\}.$$
Using the Principle of Inclusion-Exclusion plus a simple limiting argument, one can derive $$\mathbb{P} \left( \bigcap_{\ell = 0}^{d-2} \left(\frac{(d-1)!(dn-2T_{\ell})}{\ell!(\ln(n))^{d-1-\ell}} \leq r_{\ell}\right)\right) = \mathbb{P} \left( \bigcap_{\ell = 0}^{d-2} \left(T_{\ell} \geq i(r_{\ell},\ell)\right)\right)  \to \prod_{\ell=0}^{d-2} (1 - e^{-r_{\ell}}),$$ hence the $d-1$-dimensional random vector with entries $V_{n}^{(\ell)} = \frac{(d-1)!(dn-2T_{\ell})}{\ell!(\ln(n))^{d-1-\ell}}$ converges in distribution to the product of $d-1$ independent exponential variables of mean 1.
    
\end{proof}

The rest of this section is for proving the first statement of Theorem \ref{Theorem: 2nd phase bounds} (for some fixed $k$ using induction), and Section \ref{Section: Final} will be for proving the second statement (again, for some fixed $k$ using induction, assuming the first statement holds for the same $k$). Hence, for the rest of the paper we will fix some $k \in \{0\} \cup [d-2]$. \\

First, we note that (\ref{equ: 2nd phase Theorem bound}) holds w.h.p. for all $j < k$ by a simple argument: by induction on the second statement of Theorem \ref{Theorem: 2nd phase bounds}, w.h.p. if $i \in I_k$ then $S_{i}^{(j)} = 0$. By (\ref{equ: y approx}) and by definition of $E_j(i)$, if $i \in I_k$ then $ns_{j}(i/n) \ll E_{j}(i)$, completing the argument. 

Next, we prove that (\ref{equ: 2nd phase Theorem bound}) holds for $j = d-1$ if it holds for all other values of $j$: by combining (\ref{equ: Y sum}) and (\ref{equ: y explicit}), we have

\vspace{-5mm}

\begin{align*}
    \left|S^{(d-1)}_{i} - n s_{d-1}\left(\frac{i}{n}\right)\right| &= \left| \sum_{j=0}^{d-2}\left(S^{(j)}_{i} - n s_{j}\left(\frac{i}{n}\right)\right) \right| \\&\leq
    \sum_{j=0}^{d-2} 4E_j(i) \qquad \text{by (\ref{equ: 2nd phase Theorem bound}) for $j \leq d-2$} \\&<
    4 E_{d-1}(i) \qquad \text{by (\ref{equ: E approx}).}
\end{align*}

Hence, for the rest of this section, we need to show the first statement of Theorem \ref{Theorem: 2nd phase bounds} for $j \in [k, d-2]$. From now on we always assume $j$ to be in this range. We will {\it also} assume that, for all $\lambda < k$, $S_{i}^{(\lambda)} = 0$ if $i \in I_k$ (which holds w.h.p. from above).

In this section we will make use of so-called {\it critical intervals}, ranges of possible values for $S^{(j)}_i$ in which we apply a martingale argument. The lower critical interval will be $$[n s_j(i/n) - 4E_j(i), ns_j(i/n) - 3E_j(i)],$$ and the upper critical interval will be $$[n s_j(i/n) + 3E_j(i), ns_j(i/n) + 4E_j(i)].$$ Our goal is to show that w.h.p. $S^{(j)}_i$ does not cross either critical interval; however, we first need to show that $S^{(j)}_i$ sits between the critical intervals at the beginning of the phase (this is the reason why $E_j(i)$ has the $2^k$ factor; it makes a sudden jump in size between phases to accomodate a new martingale process), which is the statement of our first Lemma of this section:

\begin{lem} \label{Lemma: start of phase}

W.h.p., for all $j \in [k,d-2]$  (putting $i_{after}(-1) = i_{trans}$ for convenience of notation):$$ \left|S^{(j)}_{i_{after}(k-1) + 1} - n s_j \left(\frac{i_{after}(k-1)+1}{n}\right) \right| < 3E_j(i_{after}(k-1)+1).
$$
    
\end{lem}

\begin{proof}
     First, recall that $S^{(j)}_{i+1} - S^{(j)}_{i} \in \{-2,-1,0\}$ and $|n s_{j}((i+1)/n) - n s_{j}(i/n)| \leq 2$ for any $i$ and $j$ (see paragraph above (\ref{equ: 1st phase Lipschitz})). Second, consider the change in the bound itself between $i_{after}(k-1)$ and $i_{after}(k-1)+1$: by definitions of $i_{trans}, E_{first}, E_j$, and by (\ref{equ: E approx}), we have $1 \ll E_{first}(i_{trans}) = \Theta(n^{0.64}), E_{j}(i_{trans}+1) = \omega(n^{0.69})$, and $1 \ll E_{j}(i_{after}(k-1)) \approx \frac{1}{2}(E_{j}(i_{after}(k-1)+1)$ for $k > 0$. Hence, by induction on the first statement of Theorem \ref{Theorem: 2nd phase bounds} and by Theorem \ref{Theorem: 1st phase bounds}, the statement of the Lemma follows.

\vspace{-4mm}

\end{proof}

Next, like in Section \ref{Section: First}, we define two new random variables for each $j$ and $i \in I_k$:

\vspace{-5mm}

\begin{align*}
    S^{(j)+}_{i} &:= S^{(j)}_i - n s_j(i/n) - 4 E_j(i)\\
    S^{(j)-}_{i} &:= S^{(j)}_i - n s_j(i/n) + 4 E_j(i).
\end{align*}

We also re-introduce the stopping time $T$, now defined as the first step $i \in I_k$ for which (\ref{equ: 2nd phase Theorem bound}) is {\it not} satisfied for some $j$; if (\ref{equ: 2nd phase Theorem bound}) always holds, then let $T = \infty$. Let variable name $W$ be introduced to equip this stopping time to variable $S$, i.e.:

\vspace{-5mm}

\begin{align*}
    &W^{(j)+}_{i} := \begin{cases}
        &S^{(j)+}_{i}, i < T \\
        &S^{(j)+}_{T}, i \geq T
    \end{cases}
    &W^{(j)-}_i := \begin{cases}
        &S^{(j)-}_i, i < T \\
        &S^{(j)-}_T, i \geq T.
    \end{cases}
\end{align*}

\vspace{-4mm}

Note that $W^{(j)+}_{i}$ corresponds to the upper critical interval, and $W^{(j)-}_i$ to the lower one. Furthermore, the inequality of Theorem \ref{Theorem: 2nd phase bounds} holds if and only if $W^{(j)+}_{i_{after}(k)} \leq 0$ and  $W^{(j)-}_{i_{after}(k)} \geq 0$ for each $j$ (here we must make use of our assumption that $S_{i}^{(\lambda)} = 0$ for all $\lambda < k$). The next Lemma states that, within their respective critical intervals, they are a supermartingale and submartingale respectively:\\

\vspace{-2mm}

\begin{lem} \label{Lemma:sub-sup}
    For all $i \in I_k$ and for all $j \in [k,d-2], \ \EE[W^{(j)-}_{i+1} - W^{(j)-}_i \mid G_i] \geq 0$ whenever $W^{(j)-}_i \leq  E_j(i)$, and $\EE[W^{(j)+}_{i+1} - W^{(j)+}_i \mid G_i] \leq 0$ whenever $W^{(j)+}_i \geq  -E_j(i)$.
\end{lem}

\begin{proof}

Here we just prove the first part of the Lemma; the second part follows from nearly identical calculations. By the same logic as in the proof of Lemma \ref{Lemma: first phase sub-sup} we work with $S^{(j)-}$ instead of $W^{(j)-}$ and assume that (\ref{equ: 2nd phase Theorem bound}) holds for all $j$. We also have the same expected change as in Lemma \ref{Lemma: first phase sub-sup}, except with $E_{first}(i)$ replaced with $4E_j(i)$:

\vspace{-6mm}

\begin{align*}
    \EE[S^{(j)-}_{i+1} - S^{(j)-}_i \mid G_i] &= \frac{-2 Y^{(j)}_i}{S^{(d-1)}_i} + O\left(\frac{1}{dn-2i}\right) + \frac{2y_j(i/n)}{s_{d-1}(i/n)} -\frac{d^2}{d\mu^2}\left(\frac{ns_j(\mu/n)}{2}\right)\Big|_{\mu = \psi} \\&+
    4(E_{j}(i+1) - E_j(i)).
\end{align*}

\vspace{-4mm}

We split the above expression (excluding $O\left(\frac{1}{dn-2i}\right)$) into three summands, assuming \newline $S^{(j)-}_i \leq E_j(i) \ \Longleftrightarrow \ S^{(j)}_i - ns_j(i/n) \leq  - 3E_j(i)$ (for convenience, for the case $j = 0$, we put $S_i^{(j-1)}, s_{j-1}$, and $E_{j-1}$ all equal to 0):

\begin{enumerate}
    \item  \begin{align*}
        \frac{-2 Y^{(j)}_i}{S^{(d-1)}_i} + \frac{2y_j(i/n)}{s_{d-1}(i/n)} &= \frac{-2 S^{(j)}_i + 2ns_j(i/n) + 2 S^{(j-1)}_i - 2ns_{j-1}(i/n)}{n s_{d-1}(i/n)} \\ & \ \ \ + 2Y^{(j)}_i \left(\frac{1}{n s_{d-1}(i/n)} - \frac{1}{S^{(d-1)}_i}\right) \\&\geq 
        \frac{6 E_j(i) - 8 E_{j-1}(i)}{ns_{d-1}(i/n)} - \frac{8 S^{(j)}_{i} E_{d-1}(i)}{S^{(d-1)}_i (n s_{d-1}(i/n))} \nonumber \qquad \text{by (\ref{equ: 2nd phase Theorem bound})} \\&\geq
        \frac{5.9 E_j(i)}{ns_{d-1}(i/n)} - \frac{9 S^{(j)}_{i} E_{d-1}(i)}{(n s_{d-1}(i/n))^2} \nonumber \qquad \text{by (\ref{equ: E approx}), (\ref{equ: 2nd phase Theorem bound}), (\ref{equ: y approx}), and $i \leq i_{after}(k)$} \\&\geq
        \frac{5.9 E_j(i)}{ns_{d-1}(i/n)} - \frac{9 (n s_{j}(i/n)) E_{d-1}(i)}{(n s_{d-1}(i/n))^2} - \frac{36 E_{j}(i) E_{d-1}(i)}{(n s_{d-1}(i/n))^2} \qquad \text{by (\ref{equ: 2nd phase Theorem bound})} \\&=
        \left(\frac{E_j(i)}{n s_{d-1}(i/n)}\right)\left(5.9 - 
 9\left(\frac{s_{j}(i/n)}{s_{d-1}(i/n)}\right)^{0.3} - \frac{36*2^k \ln (n)^{0.05}}{(n s_{d-1}(i/n))^{0.3}}\right) \\&\geq
        \frac{5.8 E_j(i)}{n s_{d-1}(i/n)} \qquad \text{by $i \leq i_{after}(k)$ and (\ref{equ: y approx})}.
    \end{align*}

    \item Just as in the proof of Lemma \ref{Lemma: first phase sub-sup}:$$
        -\frac{d^2}{d\mu^2}\left(\frac{ns_j(\mu/n)}{2}\right)\Big|_{\mu = \psi} = 
        O\left(\frac{1}{dn-2i}\right). $$

\item  
    \begin{align}
        4(E_j(i+1) - E_j(i)) &=  4 \frac{dE_j(\mu)}{d\mu}\Big|_{\mu = \phi} \qquad \text{for some $\phi \in [i, i+1]$} \nonumber \\&=
    (4)(2^k) \ln (n)^{0.05} \left(\frac{0.7}{(ns_j(\phi/n))^{0.3}}\right)\left(\frac{-2y_j(\phi/n)}{s_{d-1}(\phi/n)}\right) \qquad \text{by (\ref{equ: y deriv})} \nonumber \\& =
    (4 + o(1))(2^k) \ln (n)^{0.05} \left(\frac{0.7}{(ns_j(\phi/n))^{0.3}}\right)\left(\frac{-2s_j(\phi/n)}{s_{d-1}(\phi/n)}\right) \ \ \ \ \text{by (\ref{equ: y approx initial})} \nonumber \\&=
        \frac{-(5.6 + o(1)) E_j(\phi)}{n s_{d-1}(\phi/n)} = \frac{-(5.6 + o(1))E_j(i)}{n s_{d-1}(i/n)}. \label{equ: E dif second}
    \end{align}

\end{enumerate}

Now we put the above bounds together (using (\ref{equ: y approx}), (\ref{equ: E approx}), and $i \leq i_{after}(k) \leq i_{after}(j)$): 

\vspace{-5mm}

\begin{align*}
    \EE[S_{i+1}^{(j)-} - S_{i}^{(j)-} \mid G_i] &\geq \frac{0.01E_j(i)}{n s_{d-1}(i/n)} + O\left(\frac{1}{dn-2i}\right) 
    \\&\geq 0.
\end{align*}

\vspace{-10mm}

\end{proof}

We introduce the next Lemma to get sufficiently small bounds on the one-step changes in each time step (this is known as the ``bounded hypothesis" from \cite{DiffQMethod}):

\begin{lem} \label{Lemma:bound hyp}
    For all $i \in I_k$ and all $j \in [k,d-2]$,$$
       -3 < W^{(j) \xi}_{i+1} - W^{(j)\xi}_i < \ln(n)^{-d+1.06+j}$$ where ``$\xi$" can be either ``$+$" or ``$-$".
\end{lem}

\begin{proof}

Like in the proofs of Lemma \ref{Lemma: first phase sub-sup} and \ref{Lemma:sub-sup}, we assume that $W^{\xi} = S^{(j)\xi}$ ($\xi$ is $+$ or $-$), else $W^{\xi}_{i+1} - W^{\xi}_{i} = 0$. Again, we have $-2 \leq S^{(j)}_{i+1} - S^{(j)}_{i} \leq 0$. Secondly, we have

\vspace{-5mm}

\begin{align*}
    & |-n s_j((i+1)/n)+ns_j(i/n) - C E_j(i+1)+C E_j(i)| \\ & \leq 
    |-n s_j((i+1)/n)+ns_j(i/n)| + |- C E_j(i+1)+C E_j(i)| \\ &= 
    O \left(\frac{y_j(i/n)}{s_{d-1}(i/n)} + \frac{E_j(i)}{n s_{d-1}(i/n)} \right) \qquad \text{by (\ref{equ: y deriv}), (\ref{equ: y approx}), and (\ref{equ: E dif second})} \\ &=  
    o(\ln(n)^{-d+1.06+j}) \qquad \text{by (\ref{equ: y approx}), (\ref{equ: E approx}), and $i \leq i_{after}(k) \leq i_{after}(j).$}
\end{align*}

Combining the inequalities completes the proof.

\end{proof}

To put this all together to prove the first part of Theorem \ref{Theorem: 2nd phase bounds}, we introduce a series of events: first, let $\mathcal{E}^{(j)+}$ denote the event that $W^{(j)+}_{i_{after}(k)} > 0$ and $\mathcal{E}^{(j)-}$ denote the event that $W^{(j)-}_{i_{after}(k)} < 0$. Let $\mathcal{E} = \left(\bigcup_{j \geq k}\mathcal{E}^{(j)+}\right) \cup \left(\bigcup_{j \geq k}\mathcal{E}^{(j)-}\right)$; we seek to bound $\mathbb{P}[\mathcal{E}]$, since $\mathcal{E}$ is the event that (\ref{equ: 2nd phase Theorem bound}) \textit{doesn't} hold for some $i \in I_k$. Next, for all $\ell \in I_k$, let $\mathcal{H}^{(j)+}_{\ell}$ be the event that $W^{(j)+}_{\ell - 1} < -E_j(\ell - 1)$ and $W^{(j)+}_{\ell} \geq -E_j(\ell)$, and let $$
    \mathcal{E}^{(j)+}_{\ell} := \mathcal{H}^{(j)+}_{\ell} \cap \{W^{(j)+}_i \geq -E_j(i) \text{ for all $i \geq \ell$}\} \cap \{W^{(j)+}_{i_{after(k)}} > 0\}.$$
(see Figure 1 for a visual representation of event $\mathcal{E}^{(j)+}_{\ell}$)

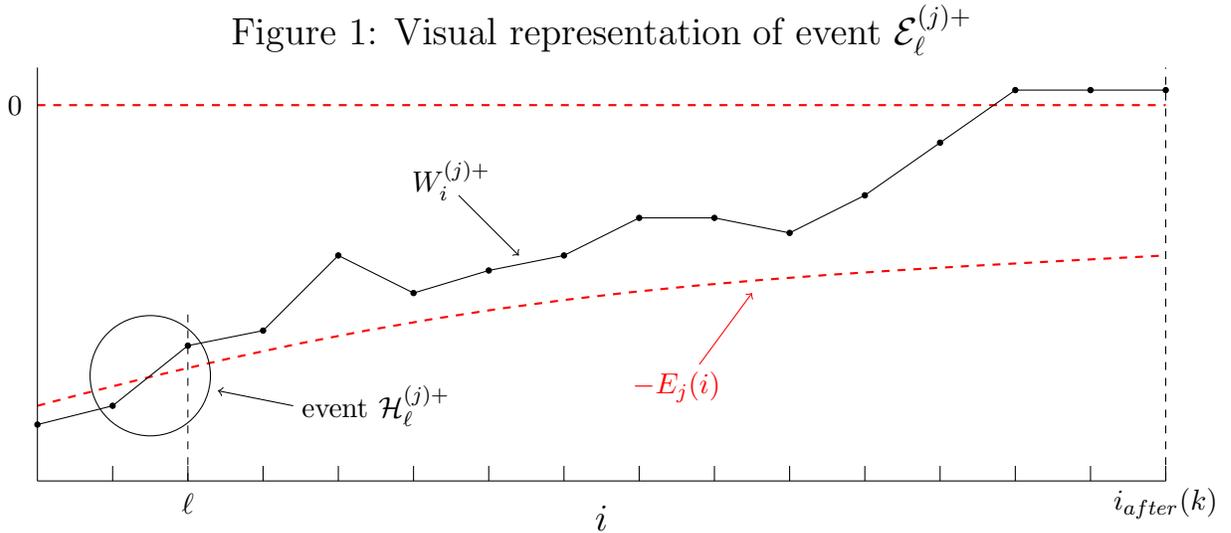
\vspace{-5mm}

\begin{center}
\begin{tikzpicture}[scale = 1]

\node (Title) at (7.5,1) {\Large Figure 1: Visual representation of event $\mathcal{E}^{(j)+}_{\ell}$};

\draw[thick, red, dashed] (0,0) -- (15,0);
\draw[thick, red, dashed] (0,-4) to[out=15,in=183] (15,-2);

\node (p1) at (0,-4.25){};
\node (p2) at (1,-4){};
\node (p3) at (2,-3.2){};
\node (p4) at (3,-3){};
\node (p5) at (4,-2){};
\node (p6) at (5,-2.5){};
\node (p7) at (6,-2.2){};
\node (p8) at (7,-2){};
\node (p9) at (8,-1.5){};
\node (p10) at (9,-1.5){};
\node (p11) at (10,-1.7){};
\node (p12) at (11,-1.2){};
\node (p13) at (12,-0.5){};
\node (p14) at (13,0.2){};
\node (p15) at (14,0.2){};
\node (p16) at (15,0.2){};

\draw (0,-5) -- (0,0.5);
\draw (0,-5) -- (15,-5);

\foreach \x in {1,...,15}
    \draw[black] ($(\x, -5)$) -- ($(\x, -4.8)$);

\filldraw
  (p1) circle (1pt)
  (p2) circle (1pt)
  (p3) circle (1pt)
  (p4) circle (1pt)
  (p5) circle (1pt)
  (p6) circle (1pt)
  (p7) circle (1pt)
  (p8) circle (1pt)
  (p9) circle (1pt)
  (p10) circle (1pt)
  (p11) circle (1pt)
  (p12) circle (1pt)
  (p13) circle (1pt)
  (p14) circle (1pt)
  (p15) circle (1pt)
  (p16) circle (1pt);

\draw ($(p1)$) -- ($(p2)$) -- ($(p3)$) -- ($(p4)$) -- ($(p5)$) -- ($(p6)$) -- ($(p7)$) -- ($(p8)$) -- ($(p9)$) -- ($(p10)$) -- ($(p11)$) -- ($(p12)$) -- ($(p13)$) -- ($(p14)$) -- ($(p15)$) -- ($(p16)$);

\node (W) at (5.5,-1){$W^{(j)+}_i$};
\draw[->] ($(W) + (0.1,-0.2)$) -- ($0.5*(p7) + 0.5*(p8) + (-0.1,0.1)$);

\node (i) at (7.5,-5.5){\Large $i$};
\node (O) at (-0.3,0){0};
\node (I) at (8.5,-3.75){\CR{$-E_j(i)$}};
\draw[->, red] ($(I) + (0.3,0.3)$) -- (9.5,-2.5);

\draw[dashed] (2,-5) -- ($(p3) + (0,0.5)$);
\node (1st) at (2,-5.3){$
\ell$};

\draw[dashed] (15,-5) -- (15,0.5);
\node (2nd) at (15,-5.3){$
i_{after}(k)$};

\draw ($0.5*(p2) + 0.5*(p3)$) circle [radius=0.8];
\node (H) at (4.5,-4){event $\mathcal{H}^{(j)+}_{\ell}$};
\draw[->] (3.4,-4) -- (2.4,-3.8);

\end{tikzpicture}

\end{center}

Similarly, for all $\ell \in I_k$, let $\mathcal{H}^{(j)-}_{\ell}$ be the event that $W^{(j)-}_{\ell - 1} > E_j(\ell - 1)$ and $W^{(j)-}_{\ell } \leq E_j(\ell)$, and let 
$$\mathcal{E}^{(j)-}_{\ell} := \mathcal{H}^{(j)-}_{\ell} \cap \{W^{(j)-}_i \leq E_j(i) \text{ for all $i \geq \ell$}\} \cap \{W^{(j)-}_{i_{after}(k)} < 0\}.$$  Finally, note that, by Lemma \ref{Lemma: start of phase}, with high probability we must have $$W^{(j)+}_{i_{after}(k-1)+1} < -E_j(i_{after}(k-1)+1) \ \text{ and } \ W^{(j)-}_{i_{after}(k-1)+1} > E_j(i_{after}(k-1)+1).$$ Furthermore, assuming these two inequalities hold (and, once again, assuming that $S^{\lambda}_i = 0$ if $\lambda < k$), then if $W^{(j)+}_{i_{after}(k)} > 0$ for some $j$, one of the events $\mathcal{E}^{(j)+}_{\ell}$ must happen; likewise, if $W^{(j)-}_{i_{after}(k)} < 0$ for some $j$, one of the events $\mathcal{E}^{(j)-}_{\ell}$ must happen; hence, $\cE^{(j)+} = \ds\bigcup_{\ell} \mathcal{E}^{(j)+}_{\ell}$ and $\cE^{(j)-} = \ds\bigcup_{\ell} \mathcal{E}^{(j)-}_{\ell}$.

We are now ready to prove the first statement of Theorem \ref{Theorem: 2nd phase bounds} in full.

\begin{proof}[Proof of the first part of Theorem \ref{Theorem: 2nd phase bounds} with fixed $k$] 

First, we fix an arbitrary $j$ (in $[k,d-2]$). We prove that $\mathbb{P}[\mathcal{E}^{(j)-}] = \exp\{-\Omega(\ln(n)^{0.036})\}$; the proof for bounding $\mathbb{P}[\mathcal{E}^{(j)+}]$ is nearly identical. We will use Corollary \ref{Cor: lower bound mod}  to bound $\mathbb{P}[\mathcal{E}^{(j)-}_{\ell}]$ for each fixed $\ell$. Given a fixed $\ell$, we define a modified stopping time $$T_{mod} := \min_{i \in [\ell,i_{after}(k)]} \{W^{(j)-}_i > E_j(i) \text{ or } i = T \}$$ (letting $T_{mod} = \infty$ if the condition doesn't hold for any $i$ in the range). Let variable $W^{\ell}_i$ be the variable $W^{(j)-}_{i}$ defined just on $i \in [\ell, i_{after}(k)]$ equipped with this stopping time (we drop the ``$(j)-$" here for convenience); i.e., $$
    W^{\ell}_i := \begin{cases}
        &W^{(j)-}_i, i < T_{mod} \\
        &W^{(j)-}_{T_{mod}}, i \geq T_{mod}.
    \end{cases} $$
Note that $(W^{\ell}_i)_i$ (over $i \in [\ell,i_{after}(k)]$) is a submartingale by Lemma \ref{Lemma:sub-sup}, since our new stopping time negates the need for the condition $W^{(j)-}_i \leq E_j(i)$; also, $(W^{\ell}_i)_i$ satisfies Lemma \ref{Lemma:bound hyp}. Since we want an upper bound for $\mathbb{P}[\mathcal{E}^{(j)-}_{\ell}]$, we can condition on event $\mathcal{H}^{(j)-}_{\ell}$, as $\mathcal{H}^{(j)-}_{\ell} \supseteq \mathcal{E}^{(j)-}_{\ell}$. Now let \vspace{-9mm}

\begin{align*}
    A_i &= -W_{\ell+i}^{\ell} + W^{\ell}_{\ell},\\
    \eta &= \ln(n)^{-d+1.06+j}, \\
    N &= 3, \\
    m &= i_{after}(k) - \ell,\\
    a &= 0.9 E_j(\ell).
\end{align*}

Note that the conditions of Corollary \ref{Cor: lower bound mod} are satisfied: $0 = A_0$ and $\eta < N/10$ are obvious, Lemma \ref{Lemma:bound hyp} gives us $-\eta \leq A_{i+1} - A_{i} \leq N$, and $(A_i)_i$ is a supermartingale since $(W^{\ell}_i)_i$ is a submartingale. We therefore implement Corollary \ref{Cor: lower bound mod}, using $m \leq \frac{dn}{2} - \ell \leq d n s_{d-1}(\ell/n)$ (by (\ref{equ: s_{d-1} bound})), (\ref{equ: y approx}), and (\ref{equ: E approx}): \vspace{-5mm}

\begin{equation}
    \mathbb{P}[A_m \geq a] \leq e^{-\frac{a^2}{3\eta N m}}+ e^{-\frac{a}{6N}} = e^{-\Omega( \ln(n)^{0.04}(n s_j(\ell/n))^{0.4})} \label{equ: apply mart inequality} + e^{-\Omega(\ln(n)^{0.05}(n s_j(\ell/n))^{0.7})}.
\end{equation}

To bound $\mathbb{P}[\mathcal{E}^{(j)-}_{\ell}]$, we show that $\mathcal{E}^{(j)-}_{\ell} \subseteq \{ A_m \geq a\}$ and apply (\ref{equ: apply mart inequality}) while conditioning on $\mathcal{H}^{(j)-}_{\ell}$. Given $\mathcal{H}^{(j)-}_{\ell}$ happens, we have $W^{\ell}_{\ell} = W^{(j)-}_{\ell} > 0.9 E_j(\ell) = a$ by (\ref{equ: E approx}), Lemma \ref{Lemma:bound hyp}, and $i \leq i_{after}(j)$. Therefore $\mathcal{E}^{(j)-}_{\ell} =  \mathcal{H}^{(j)-}_{\ell} \cap \{ W^{\ell}_{i_{after}(k)} < 0 \} \subseteq \{A_m \geq a\}$, hence $$
\mathbb{P}\left[\mathcal{E}^{(j)-}_{\ell}\right] =  e^{-\Omega( \ln(n)^{0.04}(n s_j(\ell/n))^{0.4})} + e^{-\Omega(\ln(n)^{0.05}(n s_j(\ell/n))^{0.7})}.$$

We now take a union bound to bound $\mathbb{P}[\mathcal{E}^{(j)-}]$ (using (\ref{equ: y approx}) where appropriate):

\begin{align*}
    \mathbb{P}[\mathcal{E}^{(j)-}] &\leq \sum_{\ell = i_{after}(k-1)+1}^{i_{after}(k)} \mathbb{P}\left[\mathcal{E}^{(j)-}_{\ell}\right] \\ = & \sum_{\ell=i_{trans}}^{i_{after}(k)}  \left(\exp\left\{-\Omega(\ln(n)^{0.04}(n s_j(\ell/n))^{0.4})\right\} + \exp\left\{-\Omega(\ln(n)^{0.05}(n s_j(\ell/n))^{0.7})\right\} \right)\\ = & \sum_{\ell=i_{trans}}^{i_{after}(j)}  \left(\exp\left\{- \Omega \left(\frac{(dn-2\ell)^{0.4}}{\ln(n)^{0.4d-0.44-0.4j}}\right)\right\} + \exp\left\{- \Omega \left(\frac{(dn-2\ell)^{0.7}}{\ln(n)^{0.7d-0.75-0.7j}}\right)\right\}\right) \\ = &
    \sum_{p =\lfloor\ln (n)^{d-1.01-j}\rfloor}^{\lceil n^{1 - 1/(100d)}\rceil} \left( \exp \left\{- \Omega \left(\frac{p^{0.4}}{\ln(n)^{0.4d-0.44-0.4j}}\right)\right\} + \exp \left\{- \Omega \left(\frac{p^{0.7}}{\ln(n)^{0.7d-0.75-0.7j}}\right)\right\} \right)\\ = &
    \ln(n)^{d-1.01-j}\sum_{q=1}^{\infty} \left(\exp\left\{-\Omega (q^{0.4} \ln(n)^{0.036})\right\} + \exp\left\{-\Omega (q^{0.7} \ln(n)^{0.043})\right\}\right) \\=&
    \exp\{-\Omega(\ln(n)^{0.036})\}. 
\end{align*}

We give a note for the aspects of the proof of bounding $\mathbb{P}[\mathcal{E}^{(j)+}]$ that are different from the above: use the variable $W^{(j)+}_i$ instead of $W^{(j)-}_i$, events $\mathcal{E}^{(j)+}_{\ell}$ instead of $\mathcal{E}^{(j)-}_{\ell}$, and $\mathcal{H}^{(j)+}_{\ell}$ instead of $\mathcal{H}^{(j)-}_{\ell}$. Define $T_{mod}$ instead as $$
    T_{mod} := \min_{i \in [\ell,i_{after}(k)]} \{W^{(j)+}_i < -E_j(i) \text{ or } i = T \}.$$

Finally, use Corollary $\ref{Cor: upper bound mod}$ instead of Corollary $\ref{Cor: lower bound mod}$ (which will be slightly easier to implement).

\end{proof}

\section{Final phase} \label{Section: Final}

We continue our proof by induction of Theorem \ref{Theorem: 2nd phase bounds} with our fixed index $k$; now we prove the second part. We assume the first part of Theorem \ref{Theorem: 2nd phase bounds} to hold, as well as the second part of the Theorem for lesser $k$; for example, we have $S^{(k-1)}_{i_{after}(k-1)} = 0$ w.h.p. In this section we focus on the $d$-process for a narrow domain of $i$. Let $$
    i_{before}(k) := \left\lfloor \frac{dn}{2} - \ln(n)^{d-0.8-k} \right\rfloor.$$
We will consider the $d$-process starting at step $i_{before}(k)$ assuming that (\ref{equ: 2nd phase Theorem bound}) holds at $i = i_{before}(k)$; we do not need the first part of Theorem \ref{Theorem: 2nd phase bounds} in this section otherwise. We do not use martingale arguments here, but rather we show that the distribution of the sequence of time steps at which a vertex of degree $k$ is chosen from the $d$-process is similar to a uniform distribution over all possible such sequences. Theorem \ref{Theorem: 2nd phase bounds}, (\ref{equ: y approx initial}), and (\ref{equ: E approx}) tell us that w.h.p. we will have $\sim \frac{2(d-1)!}{k!}\ln(n)^{0.2}$ vertices of degree at most $k$ (or degree equal to $k$; they are the same here) left when there are $\lfloor \ln(n)^{d-0.8-k} \rfloor$ steps left; hence, the average distance between steps at which we remove vertices of degree $k$ is $\frac{k!}{2(d-1)!} \ln(n)^{d-1-k}$. When there are this many steps left times $r$, we expect $r$ such vertices to remain, and for the probability that there are no vertices of degree $k$ to be $e^{-r}$. Most of this section will build towards proving the following Theorem:

\begin{thm} \label{theorem: last}
    Let $L(n)$ be an integer-valued function so that $L(n) = \Theta(\ln (n)^{0.2})$ and let $J(n) = \lfloor \frac{dn}{2}\rfloor - i_{before}(k) \sim \ln(n)^{d-0.8-k}$. Let $H$ be {\it any} graph with $i_{before}(k)$ edges which satisfies (\ref{equ: 2nd phase Theorem bound}) at $i = i_{before}(k)$, has no vertices of degree at most $k-1$, and has $L(n)$ vertices of degree $k$. Also, let $r \in \RR^+$ be arbitrary. Then $$\Pr\left[S^{(k)}_ {\lfloor\frac{dn}{2} - \frac{r J(n)}{L(n)} \rfloor} = 0 \ \bigg| \ G_{i_{before}(k)} = H\right] \to e^{-r}.$$
\end{thm}

First, we note that, given that (\ref{equ: 2nd phase Theorem bound}) holds for $i = i_{before}(k)$ and by (\ref{equ: Y sum}), that w.h.p. \newline $dn - 2i_{before}(k) - S^{(d-1)}_{i_{before}(k)} = O\left(\frac{dn-2i_{before}}{\ln(n)}\right)$ (consider $S^{(d-2)}_{i_{before(k)}}$); hence, for all \newline $i \in [i_{before}(k),i_{after}(k)]$:

\begin{equation}    
S^{(d-1)}_{i} = dn - 2i + 
O\left(\frac{dn-2i_{before}}{\ln(n)}\right) = (dn - 2i) \left(1 + O\left(\frac{1}{\ln(n)^{0.79}}\right) \right).\label{equ: S^{(d-1)} close}
\end{equation}

Let $t_{start} =  i_{before}(k)$ and $t_{end} = \lfloor dn/2 -  r J(n)/L(n) \rfloor$.  Consider the $d$-process between $t_{start}$ and $t_{end}$, given that $G_{t_{start}} = H$. At each step two vertices are chosen; now assume that the pair at each step is ordered uniformly at random, so that a sequence of $2(t_{end} - t_{start})$ vertices is generated. We also generate a \textit{binary} sequence simultaneously, each digit corresponding to a vertex: after a pair of vertices is picked for the vertex sequence, for each of the two vertices (in the order that they are randomly shuffled) append a ``1" to the binary sequence if the corresponding vertex had degree $k$ just before it was picked, and append a ``0" otherwise. Let $\mathcal{P}: \{0,1\}^{2(t_{end}-t_{start})} \to [0,1]$ be the corresponding probability function that arises from this process (note that, if $\gamma$ is a string with more than $L(n)$ ``1"'s, then $\mathcal{P}(\gamma)$ = 0). Note that $\mathcal{P}$ depends on the graph $H$. We compare this to a second probability function $\mathcal{Q}: \{0,1\}^{2(t_{end}-t_{start})} \to [0,1]$, which is defined by picking a binary string with $L(n)$ 1's and $2 J(n) - L(n)$ 0's uniformly at random, then taking the first $2(t_{end} - t_{start})$ digits. \\

For any binary sequence $\gamma$ with $\ell$ digits, and $I \subset [\ell]$, let $\gamma_{I}$ be the subsequence with indices from $I$; for example, $\gamma_{[a]}$ would be the first $a$ digits of $\gamma$, and $\gamma_{ \{a\}}$ would just be the $a$-th digit (for notation's sake, let ``$\gamma_{[0]}$" be the empty string). Also let $\|\gamma \|$ denote the number of 1's in $\gamma$. We now present the following Lemma:

\begin{lem} \label{conditional string}
    Let $\alpha$ be an arbitrary $2(t_{end} - t_{start})$ length binary sequence with at most $L(n)$ 1's, and let $\gamma$ be the random binary sequence according to either  $\mathcal{P}$ or $\mathcal{Q}$. Let $i \in [2(t_{end} - t_{start})]$. Then (letting $a_{\{0\}} = 1$ for sake of notation): $$
        \frac{\mathbb{P}_{\mathcal{P}}[\gamma_{[i]} = \alpha_{[i]} \mid \gamma_{[i-1]} = \alpha_{[i-1]}]}{\mathbb{P}_{\mathcal{Q}}[\gamma_{[i]} = \alpha_{[i]} \mid \gamma_{[i-1]} = \alpha_{[i-1]}]} \begin{cases}
  = 1 + O\left(\frac{1}{J(n)\ln(n)^{0.39}}\right) &\text{ if } \alpha_{\{i\}} = 0 \text{ and } \alpha_{\{i-1\}} = 0 \\
  = 1 + O\left(\frac{\ln(n)^{0.4}}{J(n)}\right) &\text{ if } \alpha_{\{i\}} = 0 \text{ and } \alpha_{\{i-1\}} = 1 \\
  = 1 + O\left(\frac{1}{\ln(n)^{0.79}}\right) &\text{ if } \alpha_{\{i\}} = 1 \text{ and } \alpha_{\{i-1\}} = 0 \\
  \leq 1 + O\left(\frac{1}{\ln(n)^{0.79}}\right) &\text{ if } \alpha_{\{i\}} = 1 \text{ and } \alpha_{\{i-1\}} = 1.
\end{cases}$$
\end{lem}

\begin{proof}
    First, we consider the cases where $\alpha_{\{i \}} = 1$. We have

\begin{equation} \label{Q string}
    \mathbb{P}_{\mathcal{Q}}[\gamma_{\{i\}} = 1 \mid \gamma_{[i-1]} = \alpha_{[i-1]}] = \frac{L(n) - \| \alpha_{[i-1]} \|}{2J(n) - (i-1)}.
\end{equation}

For the probability space $\mathcal{P}$, we need to consider three subcases: we need to consider whether $i$ is even or odd, and if it is even, whether $\alpha_{ \{i-1\} }$ is 0 or 1, since each step of the $d$-process outputs two digits of the binary string. Let's say that $\tau$ corresponds to the last step in the $d$-process before the $i$-th binary digit is generated (recall that pairs of digits are generated together). Then if $i$ is odd:

\vspace{-5mm}

\begin{align}
    \mathbb{P}_{\mathcal{P}}[\gamma_{\{i\}} = 1 \mid \gamma_{[i-1]} = \alpha_{[i-1]}] \nonumber &= -\frac{1}{2}\EE[S^{(k)}_{\tau+1} - S^{(k)}_{\tau} | G_{\tau}] \\&= \frac{S^{(k)}_{\tau}}{S^{(d-1)}_{\tau}}\left(1 + O\left(\frac{1}{dn-2{\tau}}\right) \right) \nonumber \qquad \text{by (\ref{equ: mult error})}\\&= \frac{S^{(k)}_{\tau}}{2 \lfloor dn/2-\tau \rfloor}\left(1 + O\left(\frac{1}{\ln(n)^{0.79}}\right)\right) \nonumber \qquad \text{by (\ref{equ: S^{(d-1)} close})} \\&= \frac{L(n) - \| \alpha_{[i-1]} \|}{2J(n) - (i-1)}\left(1 + O\left(\frac{1}{\ln(n)^{0.79}}\right)\right). \label{equ: i odd}
\end{align} 

If $i$ is even and $\alpha_{ \{i-1\} } = 1$, then $S^{(k)}_{\tau} = L(n) - \| \alpha_{[i-1]} \| + 1$. At step $\tau$ there are \newline $S^{(k)}_{\tau} (S^{(d-1)}_{\tau} + O(1))$ ordered pairs of vertices whose first vertex has degree $k$, and {\it at most} $2\binom{S^{(k)}_{\tau}}{2}$ ordered pairs of vertices both with degree $k$; hence:

\vspace{-5mm}

\begin{align}
    \mathbb{P}_{\mathcal{P}}[\gamma_{\{i\}} = 1 \nonumber \mid \gamma_{[i-1]} = \alpha_{[i-1]}] &\leq \frac{S^{(k)}_{\tau} - 1}{S^{(d-1)}_{\tau} + O(1)}\\&= 
    \frac{S^{(k)}_{\tau} - 1}{2 \lfloor dn/2-\tau \rfloor}\left(1 + O\left(\frac{1}{\ln(n)^{0.79}}\right)\right) \qquad \text{by (\ref{equ: S^{(d-1)} close})} \nonumber \\&=
    \frac{L(n) - \| \alpha_{[i-1]} \|}{2J(n) - (i-1)}\left(1 + O\left(\frac{1}{\ln(n)^{0.79}}\right)\right). \label{equ: i even prev 1}
\end{align} 
Hence the final inequality of the Lemma holds by (\ref{equ: i odd}) and (\ref{equ: i even prev 1}). \\

Next, consider the case where $i$ is even and $\alpha_{ \{i-1\} } = 0$; here, $S^{(k)}_{\tau} = L(n) - \| \alpha_{[i-1]} \|$ once again. At step $\tau$ there are $(S^{(d-1)}_{\tau} - S^{(k)}_{\tau}) (S^{(d-1)}_{\tau} + O(1))$ ordered pairs of vertices whose first vertex has degree greater than $k$, and $S^{(k)}_{\tau} (S^{(d-1)}_{\tau} - S^{(k)}_{\tau} + O(1))$ ordered pairs of vertices for which the first vertex has degree greater $k$ and the second vertex has degree $k$ (one can ``pick the second vertex first" to see this). Hence:

\vspace{-5mm}

\begin{align}
    \mathbb{P}_{\mathcal{P}}[\gamma_{\{i\}} = 1 \nonumber \mid \gamma_{[i-1]} = \alpha_{[i-1]}] &= \frac{S^{(k)}_{\tau} }{S^{(d-1)}_{\tau}} \left(1 + O \left(\frac{1}{S^{(d-1)}_{\tau}}\right)\right) \qquad \text{since $S^{(d-1)}_{\tau} \gg S^{(k)}_{\tau}$ here}
    \\&=
    \frac{S^{(k)}_{\tau} }{2 \lfloor dn/2-\tau \rfloor} \left(1 + O \left(\frac{1}{\ln(n)^{0.79}}\right)\right) \nonumber \qquad \text{by (\ref{equ: S^{(d-1)} close})}
    \\&=
    \frac{L(n) - \| \alpha_{[i-1]} \|}{2J(n) - (i-1)}\left(1 + O\left(\frac{1}{\ln(n)^{0.79}}\right)\right), \label{equ: i even prev 0}
\end{align} 

hence the third equality of the Lemma holds by (\ref{equ: i odd}) and (\ref{equ: i even prev 0}).\\

Now consider $\alpha_{\{i\}} = 0$. By  modifying (\ref{Q string}) to accomodate $\gamma_{\{i\}} = 0$, we have

\begin{equation} \label{new Q string}
    \mathbb{P}_{\mathcal{Q}}[\gamma_{\{i\}} = 0 \mid \gamma_{[i-1]} = \alpha_{[i-1]}] = 1 - \frac{L(n) - \| \alpha_{[i-1]} \|}{2J(n) - (i-1)}.
\end{equation}

Similarly, by modifying (\ref{equ: i odd}) and (\ref{equ: i even prev 0}), if $a_{\{i-1\}} = 0$ then

\begin{equation} \label{equ: new P string a = 0}
    \mathbb{P}_{\mathcal{P}}[\gamma_{\{i\}} = 0 \mid \gamma_{[i-1]} = \alpha_{[i-1]}]
    =
    1 - \frac{L(n) - \| \alpha_{[i-1]} \|}{2J(n) - (i-1)}\left(1 + O\left(\frac{1}{\ln(n)^{0.79}}\right)\right).
\end{equation}

By modifying (\ref{equ: i odd}) and (\ref{equ: i even prev 1}), if $a_{\{i-1\}} = 1$, then

\vspace{-5mm}

\begin{align}
    \mathbb{P}_{\mathcal{P}}[\gamma_{\{i\}} = 0 \mid \gamma_{[i-1]} = \alpha_{[i-1]}]
    & \geq
    1 - \frac{L(n) - \| \alpha_{[i-1]} \|}{2J(n) - (i-1)}\left(1 + O\left(\frac{1}{\ln(n)^{0.79}}\right)\right) \nonumber \\&=
    1 + O\left(\frac{L(n) - \| \alpha_{[i-1]} \|}{2J(n) - (i-1)}\right) \label{equ: new P string a = 1}.
\end{align}

Since $L(n) = \Theta(\ln(n)^{0.2})$ and $$2J(n) - (i-1) = 2 \lfloor dn/2 - \tau \rfloor = \Omega(dn/2 - t_{end}) = \Omega(J(n)/\ln(n)^{0.2}),$$ then $\frac{L(n) - \| \alpha_{[i-1]} \|}{2J(n) - (i-1)} = O\left(\frac{\ln(n)^{0.4}}{J(n)}\right)$. Therefore the ratio of (\ref{equ: new P string a = 0}) and (\ref{new Q string}) is $1 + O \left(\frac{1}{J(n) \ln(n)^{0.39}}\right)$, verifying the first inequality of the Lemma, and the ratio of (\ref{equ: new P string a = 1}) and (\ref{new Q string}) is $1 + O \left(\frac{\ln(n)^{0.4}}{J(n)}\right)$, verifying the second inequality of the Lemma.
    
\end{proof}

\begin{proof}[Proof of Theorem \ref{theorem: last}] First, let $\alpha$ be an arbitrary string which satisfies the criteria in Lemma \ref{conditional string}. By using the Lemma \ref{conditional string} recursively:

\vspace{-5mm}

\begin{align}
    \frac{\mathbb{P}_{\mathcal{P}}[\gamma = \alpha]}{\mathbb{P}_{\mathcal{Q}}[\gamma = \alpha]} &\leq \exp \left\{  O\left(J(n) \frac{1}{J(n) \ln(n)^{0.39}} + L(n) \left(\frac{1}{\ln(n)^{0.79}} + \frac{\ln(n)^{0.4}}{J(n)}\right)\right) \right\}\nonumber \\&=
    1 + o(1), \label{equ: string negligible}
\end{align}

and if $\alpha$ is an arbitrary string {\it with no two consecutive 1's} which satisfies the criteria in Lemma \ref{conditional string}, then by similar logic,

\begin{equation}
    \frac{\mathbb{P}_{\mathcal{P}}[\gamma = \alpha]}{\mathbb{P}_{\mathcal{Q}}[\gamma = \alpha]} = 1 + o(1). \label{equ: string important}
\end{equation}

Let $\mathcal{C}$ be the event that $\gamma$ has two consecutive 1's; we consider $\mathbb{P}[ \, \mathcal{C} \mid G_{t_{start}} = H]$. We consider probability space $\mathcal{Q}$ first. Recall that $\alpha$ is a string that has $\sim 2J(n) = \Omega(\ln(n)^{1.2})$ characters and at most $L(n) = \Theta(\ln(n)^{0.2})$ 1's. Because $\mathcal{Q}$ is a truncation of a uniform distribution, the probability of having two consecutive 1's will be $O\left(\frac{(L(n))^2}{J(n)}\right) = O(\ln(n)^{-0.8})$. Hence, by (\ref{equ: string negligible}) we must have

\begin{equation}
    \mathbb{P}_{\mathcal{P}}[ \, \mathcal{C} \mid G_{t_{start}} = H] = o(1) \ \text{ and } \ \mathbb{P}_{\mathcal{Q}}[ \, \mathcal{C} \mid G_{t_{start}} = H] = o(1). \label{equ: negligible strings}
\end{equation}

We now combine (\ref{equ: string important}) and (\ref{equ: negligible strings}) to prove Theorem \ref{theorem: last} (for ease of notation, assume we are given $G_{t_{start}} = H$):

\vspace{-5mm}

\begin{align*}
    \mathbb{P}\left[S^{(k)}_{\lfloor \frac{dn}{2} - \frac{r J(n)}{L(n)} \rfloor} = 0 \right] &= \mathbb{P}_{\mathcal{P}}[\|\gamma\| = L(n)] \\&=
    \mathbb{P}_{\mathcal{P}}[\|\gamma\| = L(n) \ | \ \mathcal{C}]\mathbb{P}_{\mathcal{P}}[\mathcal{C}] + \mathbb{P}_{\mathcal{P}}[\|\gamma\| = L(n) \ | \ \overline{\mathcal{C}}]\mathbb{P}_{\mathcal{P}}[\overline{\mathcal{C}}] \\&=
    \mathbb{P}_{\mathcal{Q}}[||\gamma|| = L(n)] + o(1) \qquad \text{by (\ref{equ: string important}) and (\ref{equ: negligible strings})}\\&=
    \frac{\binom{2(t_{end}-t_{start})}{L(n)}}{\binom{2J(n)}{L(n)}} + o(1) \\&= \frac{\binom{2J(n) - 2(\lfloor r J(n)/L(n) \rfloor)}{L(n)}}{\binom{2J(n)}{L(n)}} + o(1) \\&=
    \left(1-\frac{r(1+ o(1))}{L(n)}\right)^{L(n)} + o(1) \\& =
    e^{-r} + o(1).
\end{align*}

\end{proof}

We can now complete the proof of the second statement of Theorem \ref{Theorem: 2nd phase bounds} at value $k$. Roughly speaking, we will use Theorem \ref{theorem: last} with $L(n) \approx \frac{2(d-1)! \ln(n)^{0.2}}{k!}$, so $\frac{dn}{2} - i(r_{k},k) \approx \frac{r_k J(n)}{L(n)}.$
First, note that $S^{(k)}_{i_{after}(k)} = 0$ (w.h.p.) comes automatically when the rest of the statement is proved (by putting $r_{\ell} = 0$ for $\ell < k$ and having $r_k \to 0$). Let $\mathcal{G}_{\ell}$ be the event that $S^{(\ell)}_{ \lfloor i(r_{\ell},\ell)\rfloor}  = 0$ and $\mathcal{G} = \bigcap_{\ell \leq k} \mathcal{G}_{\ell}$, let $\cF$ be the event that (\ref{equ: 2nd phase Theorem bound}) holds for $i = i_{before}(k)$ and $S^{(j)}_{i_{before(k)}} = 0$ for $j < k$, and let $\mathcal{A} = \mathcal{F} \cap \bigcap_{\ell < k} \mathcal{G}_{\ell}$. Also, let $$\mathcal{I} = [ns_{k}(i_{before}(k)/n) - 4E_k(i_{before}(k)),ns_{k}(i_{before}(k)/n) + 4E_k(i_{before}(k))].$$ Note that, by part 1 of 
Theorem \ref{Theorem: 2nd phase bounds}, by induction on the second part Theorem \ref{Theorem: 2nd phase bounds}, and since $i_{before}(k) > i_{after}(k-1)$, $\cF$ happens with probability $1 - o(1)$. Therefore:  

\vspace{-5mm}

\begin{align*}
    \mathbb{P}[\mathcal{G}] &= \mathbb{P}\left[\mathcal{G}_{k} \cap \mathcal{A} \right] + o(1)\\&=
    \sum_{p \in \mathcal{I}} \mathbb{P}\left[\mathcal{G} \ \bigg| \  \mathcal{A} \cap (S^{(k)}_{i_{before}(k)} = p) \right] \mathbb{P}\left[\mathcal{A} \cap (S^{(k)}_{i_{before}(k)} = p) \right] + o(1).
\end{align*}

We can now apply (\ref{equ: y approx initial}), (\ref{equ: E approx}), and Theorem \ref{theorem: last} to get $$\mathbb{P}\left[\mathcal{G} \ \bigg| \  \mathcal{A} \cap (S^{(k)}_{i_{before}(k)} = p) \right] = e^{-r_{\ell}} + o(1)$$

for $p \in \mathcal{I}$. We note that all $o(1)$ functions in the sum can be made to be the same by carefully reviewing the proof of Theorem \ref{theorem: last}. Therefore:

\vspace{-5mm}

\begin{align*}
    \mathbb{P}[\mathcal{G}] &= \sum_{p \in \mathcal{I}} (e^{-r_{\ell}} + o(1)) \mathbb{P}\left[\mathcal{A} \cap (S^{(k)}_{i_{before}(k)} = p)  \right] +o(1) \qquad  \\&=
    e^{-r_{\ell}} \sum_{p \in \mathcal{I}}\mathbb{P}\left[\mathcal{A} \cap (S^{(k)}_{i_{before}(k)} = p)  \right] + o(1) \\&=
    e^{-r_{\ell}} \mathbb{P}\left[\bigcap_{\ell < k} \mathcal{G}_{\ell}\right] + o(1) \qquad \text{by Theorem \ref{Theorem: 2nd phase bounds}} \\&=
    \exp \left\{ \sum_{\ell = 0}^{k} e^{- r_{\ell}} \right\} + o(1) \qquad \text{by induction on Theorem \ref{Theorem: 2nd phase bounds},}
\end{align*}

proving Theorem \ref{Theorem: 2nd phase bounds}.

\section*{Acknowledgements}

Thanks to my advisor Tom Bohman, who gave me much guidance and input, and to the National Science Foundation for its financial support. Competing interests: The author declares none.

\bibliographystyle{plain}

\end{document}